\documentclass[11pt]{amsart}
\usepackage{amssymb}
\usepackage{bm}

\usepackage{epsfig}
\usepackage{mathdots}
 \theoremstyle{plain}
\newtheorem{theorem}{Theorem}
\newtheorem{corollary}[theorem]{Corollary}

\newtheorem{lemma}[theorem]{Lemma}
\newtheorem{proposition}[theorem]{Proposition}
\theoremstyle{definition}
\newtheorem{definition}{Definition}
\theoremstyle{remark}

\numberwithin{equation}{section}
\numberwithin{theorem}{section}

\newcommand{\bT}{\begin{theorem}}
\newcommand{\eT}{\end{theorem}}
\newcommand{\bProp}{\begin{proposition}}
\newcommand{\eProp}{\end{proposition}}
\newcommand{\bE}{\begin{example}}
\newcommand{\eE}{\end{example}}
\newcommand{\bL}{\begin{lemma}}
\newcommand{\eL}{\end{lemma}}
\newcommand{\bP}{\begin{proof}}
\newcommand{\eP}{\end{proof}}
\newcommand{\bC}{\begin{corollary}}
\newcommand{\eC}{\end{corollary}}
\newcommand{\bD}{\begin{definition}}
\newcommand{\eD}{\end{definition}}
\newcommand{\be}{\begin{enumerate}}
\newcommand{\ee}{\end{enumerate}}
\newcommand{\beqa}{\begin{eqnarray*}}
\newcommand{\eeqa}{\end{eqnarray*}}
\newcommand{\beqaa}{\begin{eqnarray}}
\newcommand{\eeqaa}{\end{eqnarray}}
\newcommand{\ba}{\begin{array}}
\newcommand{\ea}{\end{array}}

\setbox0=\hbox{$+$}
\newdimen\plusheight
\plusheight=\ht0
\def\+{\;\lower\plusheight\hbox{$+$}\;}

\setbox0=\hbox{$-$}
\newdimen\minusheight
\minusheight=\ht0
\def\-{\;\lower\minusheight\hbox{$-$}\;}

\setbox0=\hbox{$\cdots$}
\newdimen\cdotsheight
\cdotsheight=\plusheight
\def\cds{\lower\cdotsheight\hbox{$\cdots$}}
\begin{document}
\title[Further Results on Vanishing Coefficients ]
       {Further Results on Vanishing Coefficients in Infinite Product Expansions}
       {\allowdisplaybreaks
\author{James Mc Laughlin}
\address{Mathematics Department\\
 25 University Avenue\\
West Chester University, West Chester, PA 19383}
\email{jmclaughl@wcupa.edu}}


\thanks{This work was partially supported by a grant from the Simons Foundation (\#209175 to James McLaughlin).
 }
\date{\today}

\begin{abstract}
We extend results of Andrews and Bressoud on the vanishing of coefficients in the series expansions of certain infinite products.
 These results have the form
that if \begin{equation*}
\frac{(q^{r-tk},q^{mk-(r-tk)};q^{mk})_{\infty}}{(q^{r},q^{mk-r};q^{mk})_{\infty}}=:\sum_{n=0}^{\infty} c_nq^n,
\end{equation*}
for certain integers $k$, $m$ $s$ and $t$, where $r=sm+t$, then $c_{kn-rs}$ is always zero.
Our theorems also partly give a simpler reformulation of  results of Alladi and Gordon, but also give results for cases not covered by the theorems of Alladi and Gordon.

We also give some interpretations of the analytic results in terms of integer partitions.
\end{abstract}

\maketitle

\section{Introduction and Background}

In the present paper we prove some new results on vanishing coefficients in the series expansion of certain infinite $q$-products. These results have the form
that if \begin{equation*}
\frac{(q^{r-tk},q^{mk-(r-tk)};q^{mk})_{\infty}}{(q^{r},q^{mk-r};q^{mk})_{\infty}}=:\sum_{n=0}^{\infty} c_nq^n,
\end{equation*}
(where $k$, $m$ $s$ and $t$ are integers to be defined in more detail below, such that $r=sm+t$) then $c_{kn-rs}$ is always zero.  Some new theorems on integer partitions, which follow from these analytic results, are also given.  Before coming to these new results, we first recall some prior work by previous authors on the topic.

In \cite{RS78},
Richmond and Szekeres proved that if
\begin{equation*}
F(q):=\frac{(q^3,q^5;q^8)_{\infty}}{(q,q^7;q^8)_{\infty}}=:\sum_{m=0}^{\infty}c_m q^m,
\end{equation*}
then $c_{4n+3}$ is always zero.
They also showed that if
\begin{equation*}
\frac{1}{F(q)}=:\sum_{m=0}^{\infty}d_m q^m,
\end{equation*}
then $d_{4n+2}$ is always zero. These results were derived by Richmond and Szekeres from Hardy - Ramanujan - Rademacher expansions they developed of the infinite products. They also conjectured that if
\begin{equation*}
G(q):=\frac{(q^5,q^7;q^{12})_{\infty}}{(q,q^{11};q^{12})_{\infty}}=:\sum_{m=0}^{\infty}a_m q^m,
\end{equation*}
then $a_{6n+5}$ is always zero, and if
\begin{equation*}
\frac{1}{G(q)}=:\sum_{m=0}^{\infty}b_m q^m,
\end{equation*}
then $b_{6n+3}$ is always zero.

In \cite{AB79}, Andrews and Bressoud proved the following general theorem, which generalizes the results of Richmond and Szekeres as special cases.
\begin{theorem}\label{t1}
If $1\leq r <k$ are relatively prime integers of opposite parity and
\begin{equation}\label{eq1}
\frac{(q^r,q^{2k-r};q^{2k})_{\infty}}{(q^{k-r},q^{k+r};q^{2k})_{\infty}}=:\sum_{n=0}^{\infty} \phi_nq^n,
\end{equation}
then $\phi_{kn+r(k-r+1)/2}$ is always zero.
\end{theorem}
Andrews and Bressoud derived their result from Ramanujan's $_1\psi_1$ summation formula,
\begin{equation}\label{ram1}
\sum_{n=-\infty}^{\infty} \frac{(a;q)_nz^n}{(b;q)_n}=\frac{(b/a,q,az,q/az;q)_{\infty}}{(q/a,b,z,b/az;q)_{\infty}},
\end{equation}
after replacing $q$ with $q^{2k}$, specializing $a$, $b$ and $z$ and employing some $q$-series manipulations.
The cases $(k,r)=(4,3),\,(4,1),\, (6,5)$ and $(6,1)$, respectively, give the two results proved by Richmond and Szekeres, and the two results conjectured by them.

Alladi and Gordon \cite{AG94} prove a yet more general theorem (we modify their notation to state their results  in the same language used elsewhere in the present paper).
\begin{theorem}\label{AGt1}
Let $1<m<k$ and let $(s, km)=1$ with $1\leq s < mk$. Let $r^{*}=(k-1)s$ and $r\equiv r^{*} (\mod m k)$, with $1\leq r<mk$. \\
Put $r'=\lceil \frac{r^{*}}{mk}\rceil (\mod k)$ with $1 \leq r'<k$. Write
\[
\frac{(q^{r},q^{mk-r};q^{mk})_{\infty}}{(q^{s},q^{mk-s};q^{mk})_{\infty}}
=\sum_{n=0}^{\infty}a_n q^n.
\]
Then $a_n=0$ for $n \equiv rr'(\mod k)$.
\end{theorem}
Note that while there is certainly some overlap with our Theorem \ref{t2n} below, the result of Alladi and Gordon in Theorem \ref{AGt1} does not provide any information about vanishing coefficients in the cases where $k<m$ or $k=m$. In contrast, our Theorem \ref{t2n} below has no such restrictions.

Alladi and Gordon \cite{AG94}  also prove a companion theorem to Theorem \ref{AGt1} above.
\begin{theorem}\label{AGt2}
Let $m$, $k$, $s$, $r^{*}$, $r$, $r'$ be as in Theorem \ref{AGt2} with $k$ odd. Write
\[
\frac{(q^{r},q^{mk-r};q^{mk})_{\infty}}{(-q^{s},-q^{mk-s};q^{mk})_{\infty}}
=\sum_{n=0}^{\infty}a'_n q^n.
\]
Then $a'_n=0$ for $n \equiv rr'(\mod k)$.
\end{theorem}
 We also prove a companion theorem to our  Theorem \ref{t2n}, namely  Theorem \ref{t3n} below, which is similar in nature to Theorem \ref{AGt2} of Alladi and Gordon, but as with Theorem \ref{t2n}, our result is not restricted to $k>m$, as is the case in their theorem.

 \section{Main Results}

In the present paper our main result,  in Theorem \ref{t2n} below, is in part a reformulation of Theorem \ref{AGt1} of Alladi and Gordon \cite{AG94}, but also extends to cases not covered by Theorem \ref{AGt1}. The proof of Theorem \ref{t2n} also uses Ramanujan's $_1\psi_1$ summation formula.


\begin{theorem}\label{t2n}
Let $k>1$, $m>1$ be positive integers. Let $r=sm+t$, for some integers $s$ and $t$, where $0\leq s<k$, $1\leq t <m$ and $r$ and $k$ are relatively prime. Let
\begin{equation}\label{eq2n}
\frac{(q^{r-tk},q^{mk-(r-tk)};q^{mk})_{\infty}}{(q^{r},q^{mk-r};q^{mk})_{\infty}}=:\sum_{n=0}^{\infty} c_nq^n,
\end{equation}
then $c_{kn-rs}$ is always zero.
\end{theorem}

\begin{proof}
In Ramanujan's $_1\psi_1$ summation formula \eqref{ram1},
replace $q$ with $q^{mk}$, $a$ with $q^{-tk}$, $b$ with $q^{mk-tk}$ and $z$ with $q^r$ (note that these choices satisfy the requirements needed for the series to converge, namely $|b/a|<|z|<1$, since $r<mk$). This gives, after a little simplification
\begin{equation}\label{ram2n}
-q^{-tk}\sum_{n=-\infty}^{\infty} \frac{q^{rn}}{1-q^{nmk-tk}}=\frac{(q^{mk},q^{mk};q^{mk})_{\infty}}
{(q^{tk},q^{mk-tk};q^{mk})_{\infty}}\frac{(q^{r-tk},q^{mk-(r-tk)};q^{mk})_{\infty}}
{(q^r,q^{mk-r};q^{mk})_{\infty}}.
\end{equation}
It is clear that to prove the result, all that is necessary is to show that if we expand
\begin{equation}\label{ram3n}
\sum_{n=-\infty}^{\infty} \frac{q^{rn}}{1-q^{nmk-tk}}=\sum_{n=0}^{\infty} \frac{q^{rn}}{1-q^{nmk-tk}}-\sum_{n=1}^{\infty} \frac{q^{nmk+tk-rn}}{1-q^{nmk+tk}} =:\sum_{n=o}^{\infty} d_nq^n,
\end{equation}
then $d_{kn-rs}=0$ for all $n$. Thus we just need to consider those $n$ that lead to powers of $q$ of the form $q^{kn-rs}$ in the two sums in the middle expression just above, and thus all that is necessary is to show that
\begin{equation}\label{ram4n}
\sum_{n=1}^{\infty} \frac{q^{r(nk-s)}}{1-q^{(nk-s)mk-tk}}-\sum_{n=0}^{\infty} \frac{q^{(nk+s)mk+tk-r(nk+s)}}{1-q^{(nk+s)mk+tk}} =0.
\end{equation}
Note that if $s=0$, then the first sum should start at $n=0$ and the
second sum should start at $n=1$, but it can be seen that if $s=0$,
then the term corresponding to $n=0$ in the first series is
$1/(1-q^{-tk})=-q^{tk}/(1-q^{tk})$, while the term corresponding to $n=0$
in the second series is $q^{tk}/(1-q^{tk})$ when $s=0$. Thus the assertion
that all that is necessary to prove the result is to show that
\eqref{ram4n}, also holds when $s=0$.
\begin{align*}
\sum_{n=1}^{\infty} \frac{q^{rnk-rs}}{1-q^{nmk^2-smk-tk}}&=\sum_{n=1}^{\infty}q^{rnk-rs}
\sum_{p=0}^{\infty}q^{p(nmk^2-smk-tk)}\\
&=\sum_{p=0}^{\infty}q^{p(-smk-tk)-rs}\sum_{n=1}^{\infty}q^{n(pmk^2+ rk)}\\
&=\sum_{p=0}^{\infty}\frac{q^{p(-smk-tk+mk^2)+ rk-rs}}{1-q^{pmk^2+ rk}}\\
&=\sum_{n=0}^{\infty}\frac{q^{n(-(sm+t)k+mk^2)+ rk-rs}}{1-q^{pmk^2+ rk}}.
\end{align*}
We now use the fact that $r=sm+t$, which easily implies  that the last series above and the second series at \eqref{ram4n} are identical, giving the result.
\end{proof}

Remark: It may happen $r<tk$, in which case it will be necessary to
use the identity
\[
\frac{(q^{r-tk},q^{mk-(r-tk)};q^{mk})_{\infty}}{(q^{r},q^{mk-r};q^{mk})_{\infty}}=
\frac{-1}{q^{tk-r}}\frac{(q^{mk-(tk-r)},q^{tk-r};q^{mk})_{\infty}}{(q^{r},q^{mk-r};q^{mk})_{\infty}}
\]
if it is desired that all the exponents in the infinite products be
positive.

We give the following example as an illustration of the result in
Theorem \ref{t2n}, and also to highlight the differences between this
result and that of Andrews and Bressoud in Theorem \ref{t1}. In each
case $mk=30$, $r=t=1$ and $s=0$, so that Theorem \ref{t2n} gives that
$c_{kn}=0$ for all $n$. However, since $r-k<0$ in each case, we
modify the infinite products as described above, so that the
progressions containing zero coefficients are thus shifted.
\begin{corollary}
a) Let
\[
\frac{(q^2,q^{28};q^{30})_{\infty}}{(q,q^{29};q^{30})_{\infty}}=\sum_{n=0}^{\infty}c_nq^n.
\]
Then $c_{3n+2}=0$ for all $n\geq 0$. (Here $k=3$, so $r-k=-2$.)

b) Let
\[
\frac{(q^4,q^{26};q^{30})_{\infty}}{(q,q^{29};q^{30})_{\infty}}=\sum_{n=0}^{\infty}c_nq^n.
\]
Then $c_{5n+4}=0$ for all $n\geq 0$. (Here $k=5$, so $r-k=-4$.)

c)  Let
\[
\frac{(q^5,q^{25};q^{30})_{\infty}}{(q,q^{29};q^{30})_{\infty}}=\sum_{n=0}^{\infty}c_nq^n.
\]
Then $c_{6n+5}=0$ for all $n\geq 0$. (Here $k=6$, so $r-k=-5$.)

d) Let
\[
\frac{(q^9,q^{21};q^{30})_{\infty}}{(q,q^{29};q^{30})_{\infty}}=\sum_{n=0}^{\infty}c_nq^n.
\]
Then $c_{10n+9}=0$ for all $n\geq 0$. (Here $k=10$, so $r-k=-9$.)

e) Let
\[
\frac{(q^{14},q^{16};q^{30})_{\infty}}{(q,q^{29};q^{30})_{\infty}}=\sum_{n=0}^{\infty}c_nq^n.
\]
Then $c_{15n+14}=0$ for all $n\geq 0$. (Here $k=15$, so $r-k=-14$.)

\end{corollary}

Remark: The identity at e) is also given by Theorem \ref{t1} of
Andrews and Bressoud ($k=15$ and $r=14$ in their theorem), but none
of the identities a) - d) above follow from their theorem. Similarly, parts c), d), and e) are given by Theorem \ref{AGt1} of Alladi and Gordon, but not parts a) and b).

We also give the following result with  to further illustrate  the difference between Theorem \ref{t2n}
and Theorem \ref{AGt1} of Alladi and Gordon (which does not imply the result in Corollary \ref{t2nc2}, since $k=m$). In each case in the corollary below, $k=m=3$.
\begin{corollary}\label{t2nc2}
a) Let
\[
\frac{(q,q^{8};q^{9})_{\infty}}{(q^{4},q^{5};q^{9})_{\infty}}=\sum_{n=0}^{\infty}c_nq^n.
\]
Then $c_{3n+2}=0$ for all $n\geq 0$. (Take $s=t=1$, so $r=1(3)+1=4$, $r-tk=4-1(3)=1$ and $-rs\equiv 2(\mod 3)$.)

b) Let
\[
\frac{(q^2,q^{7};q^{9})_{\infty}}{(q^{},q^{8};q^{9})_{\infty}}=\sum_{n=0}^{\infty}c_nq^n.
\]
Then $c_{3n+2}=0$ for all $n\geq 0$. (Take $s=t=2$, so $r=2(3)+2=8$, $r-tk=8-2(3)=2$ and $-rs \equiv 2 (\mod 3)$.)

c)  Let
\[
\frac{(q^4,q^{5};q^{9})_{\infty}}{(q^{2},q^{7};q^{9})_{\infty}}=\sum_{n=0}^{\infty}c_nq^n.
\]
Then $c_{3n+1}=0$ for all $n\geq 0$. (Take $s=2$, $t=1$, so $r=2(3)+1=7$, $r-tk=7-1(3)=4$ and $-rs \equiv 1 (\mod 3)$.)

\end{corollary}

There is also a companion result to Theorem \ref{t2n}, in the same way that Alladi and Gordon's Theorems \ref{AGt1} and \ref{AGt2} are companions. However, in contrast to Theorem \ref{AGt2}, our Theorem \ref{t3n} does not have the restriction that $m<k$.

\begin{theorem}\label{t3n}
Let $k>1$, $m>1$ be positive integers, with $k$ odd. Let $r=sm+t$, for some integers $s$ and $t$, where $0\leq s<k$, $1\leq t <m$ and $r$ and $k$ are relatively prime. Let
\begin{equation}\label{eq3n}
\frac{(q^{r-tk},q^{mk-(r-tk)};q^{mk})_{\infty}}{(-q^{r},-q^{mk-r};q^{mk})_{\infty}}=:\sum_{n=0}^{\infty} d_nq^n,
\end{equation}
then $d_{kn-rs}$ is always zero.
\end{theorem}

\begin{proof}
The argument is essentially the same as that used in the proof of Theorem \ref{t2n}, so details are omitted. The only additional facts needed are that if $k$ is odd, then $(-1)^{kn}=(-1)^n$, and $(-1)^{n-s}=(-1)^{nk+s}$.
\end{proof}

\begin{corollary}\label{t3nc2}
a) Let
\[
\frac{(q,q^{8};q^{9})_{\infty}}{(-q^{4},-q^{5};q^{9})_{\infty}}=\sum_{n=0}^{\infty}c'_nq^n.
\]
Then $c'_{3n+2}=0$ for all $n\geq 0$.

b) Let
\[
\frac{(q^2,q^{7};q^{9})_{\infty}}{(-q^{},-q^{8};q^{9})_{\infty}}=\sum_{n=0}^{\infty}c'_nq^n.
\]
Then $c'_{3n+2}=0$ for all $n\geq 0$.

c)  Let
\[
\frac{(q^4,q^{5};q^{9})_{\infty}}{(-q^{2},-q^{7};q^{9})_{\infty}}=\sum_{n=0}^{\infty}c'_nq^n.
\]
Then $c'_{3n+1}=0$ for all $n\geq 0$.
\end{corollary}
\begin{proof}
Let $k=m=3$ in Theorem \ref{t3n}, and let $s$ and $t$ have the same values as in the corresponding parts of Corollary \ref{t2nc2}.
\end{proof}

\section{Partition Implications}

Theorems \ref{t2n} and \ref{t3n} also have implications for  certain types of restricted partitions.
\begin{theorem}
Let $k>1$, $m>1$ be positive integers. Let $r=sm+t$, for some integers $s$ and $t$, where $0\leq s<k$, $1\leq t<m$, and $r$ and $k$ are relatively prime.
Let $p_{m,k,r}(n)$ denote the number of partitions of $n$ into parts $\equiv 0, \pm r (\mod m k)$. Then for each integer $n$,
\begin{equation*}
\sum_{j}
(-1)^jp_{m,k,r}(n k - r s - m k j (j + 1)/2 - j (tk - r))
=0,
\end{equation*}
where the sum is over those $j$ with $n k - r s - m k j (j + 1)/2 - j (tk - r) \geq 0$.
\end{theorem}

\begin{proof}
The coefficient of $q^{nk-rs}$ in
\begin{equation*}
\frac{(q^{r-tk},q^{mk-(r-tk)};q^{mk})_{\infty}}{(q^{r},q^{mk-r};q^{mk})_{\infty}}
=\frac{(q^{r-tk},q^{mk-(r-tk)},q^{mk};q^{mk})_{\infty}}{(q^{r},q^{mk-r},q^{mk};q^{mk})_{\infty}}
\end{equation*}
is zero for all $n$, by the theorem.
However
\begin{multline*}
\frac{(q^{r-tk},q^{mk-(r-tk)},q^{mk};q^{mk})_{\infty}}{(q^{r},q^{mk-r},q^{mk};q^{mk})_{\infty}}\\
=\sum_{j\in\mathbb{Z}}q^{mkj(j+1)/2}(-q^{tk-r})^j\sum_{i\geq 0}p_{m,k,r}(i)q^i\\
=\sum_{N}q^N\sum_{mkj(j+1)/2+(tk-r)j+i=N}
(-1)^jq^{mkj(j+1)/2+(tk-r)j+i}p_{m,k,r}(i)
\end{multline*}
The result now follows upon setting $N=nk-rs$ and solving for $i$.
\end{proof}
As an example, take $k=15$, $m=2$, $s=0$ and $t=1$ (so $r=1$) and $n=20$, so that
\[
n k - r s - m k j (j + 1)/2 - j (tk - r)=300-15j^2-29j.
\]
\begin{table}[ht]
\begin{center}    \begin{tabular}{| c | c | c |}    \hline
$j$ & $n_j=300-15j^2-29j$ & $(-1)^j p_{2,15,1}(n_j)$\\
\hline
-5&70 &-13 \\
-4&176&203\\
-3&252&-1654\\
-2&298&3838\\
-1&314&-5773\\
0&300&4673\\
1&256&-1654\\
2&182&393\\
3&78&-13\\\hline
& & $\sum =0$\\
\hline
\end{tabular}
 \end{center}\end{table}

Theorem \ref{t3n} similarly has an interpretation in terms   of certain restricted partition functions.
\begin{theorem}
Let $k>1$, $m>1$ be positive integers, with $k$ odd. Let $r=sm+t$, for some integers $s$ and $t$, where $0\leq s<k$, $1\leq t<m$, and $r$ and $k$ are relatively prime.

Let $p^{e}_{m,k,s,t}(n)$ denote the number of partitions of $n$ into parts (possibly repeating)  $\equiv  \pm r (\mod m k)$ and distinct parts $\equiv  \pm (r-tk) (\mod m k)$, where the total number of parts, counting multiplicities, is even.

Let $p^{o}_{m,k,s,t}(n)$ denote the number of partitions of $n$ into parts (possibly repeating)  $\equiv  \pm r (\mod m k)$ and distinct parts $\equiv  \pm (r-tk) (\mod m k)$, where the total number of parts, counting multiplicities, is odd.

a) If $r-tk>0$, then for each integer $n$,
\begin{equation*}
p^{e}_{m,k,s,t}(nk-rs)-p^{o}_{m,k,s,t}(nk-rs)
=0.
\end{equation*}
b) If $r-tk<0$, then for each integer $n$,
\begin{equation*}
p^{e}_{m,k,s,t}(nk-r(s+1))-p^{o}_{m,k,s,t}(nk-r(s+1)).
=0.
\end{equation*}

\end{theorem}

\begin{proof}
It is clear that if $r-tk>0$, then
\begin{align*}
\frac{(q^{r-tk},q^{mk-(r-tk)};q^{mk})_{\infty}}{(-q^{r},-q^{mk-r};q^{mk})_{\infty}}
&=\sum_{n=0}^{\infty} (p^{e}_{m,k,s,t}(n)-p^{o}_{m,k,s,t}(n))q^n\\
&=\sum_{n=0}^{\infty}d_n q^n,
\end{align*}
where the $d_n$ are as defined in Theorem \ref{t3n}, and part a) follows. Part b) also follows from Theorem \ref{t3n}, after writing $(q^{r-tk},q^{mk-(r-tk)};q^{mk})_{\infty}$
as $-(q^{tk-r},q^{mk-(tk-r)};q^{mk})_{\infty}/q^{tk-r}$, and then shifting the $-q^{tk-r}$ to the series side.
\end{proof}

As an example of this result, again take $k=15$, $m=2$, $s=8$ and $t=1$ (so $r=17$ and $r-tk=2>0$). Then $-rs=-126\equiv 14 (\mod 15)$, and we consider $n=9\times 15+14=149$. Then $p^{e}_{2,15,8,1}(149)=p^{o}_{2,15,8,1}(149)=6$, as indicated by the following table (each function counts partitions into distinct part $\equiv \pm2 (\mod 30)$ and possibly repeating parts $\equiv \pm 17 (\mod 30)$).
\begin{table}[ht]
\begin{center}    \begin{tabular}{| c | c | }    \hline
Partitions  counted by $p^{o}_{2,15,8,1}(149)$ & Partitions counted by $p^{e}_{2,15,8,1}(149)$ \\
\hline
$2+13+17^6+32$&$2+13^{10}+17$ \\
$2+17^7+28$&$2+13^8+43$\\
$2+17^4+32+47$&$13^8+17+28$\\
$2+17^5+62$&$13^6+28+43$\\
$13+17^8$&$13^9+32$\\
$17^6+47$&$13^7+58$\\
\hline
\end{tabular}
 \end{center}\end{table}

It might be illuminating to provide combinatorial proofs of the two partitions theorems in this section.

{\allowdisplaybreaks

}

\end{document}